\documentclass[12pt]{article}

\usepackage[T1,T2A]{fontenc}
\usepackage[utf8]{inputenc}
\usepackage[english]{babel}

\usepackage[left=2cm,right=2cm, top=2cm,bottom=2cm,bindingoffset=0cm]{geometry}

\usepackage{amsmath}
\usepackage{amssymb}
\usepackage{amsfonts}
\usepackage{amsthm}

\usepackage{graphicx}
\usepackage{amscd}                
\usepackage{multirow}
\usepackage[matrix,arrow,curve]{xy}

\usepackage{paralist}
\usepackage{enumitem}

\usepackage[colorlinks=false, urlcolor=blue]{hyperref}

\theoremstyle{plain}
\newtheorem{theorem}{Theorem}[subsection]
\newtheorem{lemma}[theorem]{Lemma}

\newtheorem*{claim*}{Claim}
\newtheorem*{theorem*}{Theorem}
\newtheorem*{lemma*}{Lemma}
\newtheorem*{observation*}{Observation}
\newtheorem*{corollary*}{Corollary}
\theoremstyle{definition}

\newtheorem*{definition*}{Definition}
\newtheorem*{problem*}{Problem}
\newtheorem*{problems*}{Problems}
\newtheorem*{fact*}{Fact}
\theoremstyle{remark}

\newtheorem*{example*}{Example}
\newtheorem*{remark*}{Remark}

\usepackage{multicol}
\usepackage{soul}

\renewcommand{\Pr}{\mathsf{P}}

\newcommand{\EE}{\operatorname{\mathsf{E}}}

\renewcommand{\geq}{\geqslant}
\renewcommand{\leq}{\leqslant}
\newcommand{\diff}{\mathop{}\!\mathrm{d}}

\renewcommand{\epsilon}{\varepsilon}

\newcommand{\Binomial}{\mathrm{Bin}}

\newcommand{\Ind}{\mathbf{1}}

\renewcommand*{\thetheorem}{%
  \ifnum\value{subsection}=0 %
    \thesection
  \else
    \thesubsection
  \fi
  .\arabic{theorem}%
}

\usepackage[perpage]{footmisc}

\begin{document}

\setlength{\parskip}{0.2cm}
\begin{center}
    \Large
    On the Random Minimum Spanning Subgraph Problem for Hypergraphs\\
    \vspace*{1cm}
    \large
    Nikita Zvonkov\footnote{The work was supported by the HSE University Basic Research Program.}\\
    \normalsize\textit{Faculty of Computer Science, HSE University}\\
    
\end{center}
    
\begin{abstract}
The weight of the minimum spanning tree in a complete weighted graph with random edge weights is a well-known problem. For various classes of distributions, it is proved that the weight of the minimum spanning tree tends to a constant, which can be calculated depending on the distribution. In this paper, we generalise this result to the hypergraphs setting.
\end{abstract}

\section{Introduction}
    Consider a complete graph $\mathfrak{G}_n$ on $n$ vertices. For any graph $G$, by $E(G)$ denote the set of its edges and by $V(G)$ denote the set of its vertices. Let $\mathfrak{G}_n$ be weighted (consider additional function $w:E(\mathfrak{G}_n)\to\mathbb{R}$, mapping any edge to its weight) with edge weights being iid\footnote{Here and later iid means \textit{independent and identically distributed}.} random variables distributed as $X, X>0$ a.s.\footnote{Here and later a.s. means \textit{almost surely} and denotes an event that occurs with probability of 1.} By $F_X(x)$ we denote the distribution function of $X$: $F_X(x)\overset{\textnormal{def}}{=}\Pr(X\leq x).$

By a \textit{spanning tree} of a graph $G$ denote a subgraph $T$ that is a tree which includes all of the vertices of $G$. \textit{Minimum spanning tree} is a spanning tree with minimal sum of its edge weights. Let MST$(G)$ denote this sum. As it was shown by A. Frieze \cite{F1985}, for some distributions of $X$ there is a limit value for MST$(\mathfrak{G}_n)$ as $n$ tends to infinity:

\begin{theorem}[A. Frieze]
    Consider $\mathfrak{G}_n$, whose edges are distributed as $X$, such that $\EE X^2<+\infty$, $F_X$ has a right derivative in 0 and $F'_X(0)=a,a\in(0,+\infty)$. Then, the following holds:
    \[\textnormal{MST}(\mathfrak{G}_n)\overset{\Pr}{\to}\frac{\zeta(3)}{a},\]
    where $\zeta(3)=\sum\limits_{k=1}^{+\infty}\frac{1}{k^3}\approx1,202056903\dots$.
\end{theorem}

Let us define the generalisation on the MST problem for hypergraphs. Consider a complete $t$-uniform (that is, every edge contains $t$ different vertices) weighted hypergraph $\tilde{\mathfrak{G}}_n$ on $n$ vertices with edge weights being iid random variables distributed as $X, X>0$ a.s. Firstly, we call a hypergraph connected if by substituting all of its edges with cliques on the same set of vertices, we get a connectd graph. Similarly, by \textit{minimum spanning subgraph} of a hypergraph $\tilde{G}$ we denote connected subgraph $H$ including all the vertices of $\tilde{G}$ with minimal sum of edge weights. For convenience, denote this weight by MST$(\tilde{G})$. In this paper we establish bounds on MST$(\tilde{\mathfrak{G}}_n)$ with random edge weights.

\begin{theorem}\label{th:main}
    Consider a complete $t$-uniform hypergraph $\tilde{\mathfrak{G}}_n$ with iid edge weights distributed as $X, X>0$ a.s. Then if $\lim\limits_{x\to 0+} \frac{x}{\sqrt[t-1]{F_X(x)}}= a,a\in(0,+\infty)$, there are constants $L_t$ and $U_t$ such that for any $\varepsilon>0$, 
\[\Pr\left(\textnormal{MST}(\tilde{\mathfrak{G}}_n)\in(L_t/a-\varepsilon, U_t/a+\varepsilon)\right)\to1.\]
These constants are, correspondingly,
\[L_t=\sqrt[t-1]{c(t-1)!}\frac{t-1}{t^2}\int\limits_0^1\left(\frac{\ln(1-x)}{\left(1-x\right)^{t-1}-1}\right)^{t/(t-1)}\left(1-(1-x)^{t-1}\right)\diff x,\]
\[U_t=\sqrt[t-1]{c(t-1)!}\frac{t-1}{t}\int\limits_0^1\left(\frac{\ln(1-x)}{\left(1-x\right)^{t-1}-1}\right)^{t/(t-1)}x^{t-1}\diff x.\]
The gap between the bounds has the order of magnitude $\displaystyle\frac{U_t}{L_t}=O(\log t)$.
\end{theorem}

We prove this theorem by constructing two algorithms, one of which returns a value not less than the actual MST and the other one makes a lower bound.
By $\mathcal{A}_1,\mathcal{A}_2$ we denote these algorithms. $\mathcal{A}_i: \mathcal{G}\to\mathbb{R}$, where $\mathcal{G}=\mathbb{R}_+^{N}$ is the set of possible edge weights and $N=\binom{n}{t}$ is the number of edges in $\tilde{\mathfrak{G}}_n$.

Consider a random hypergraph process $(G_i)_{i=0}^N=G_0\subsetneq G_1\subsetneq\dots\subsetneq G_N$, where $G_0$ is a hypergraph without any edges and $G_i$ is obtained from $G_{i-1}$ by drawing additional uniform edge $e_i$, which is not present in $G_i$.

Both the algorithms are greedy in a sense that for any $i\in\{1,2\}$, $\mathcal{A}_i(G)=\sum\limits_{i=1}^{N}\alpha_i(G_{i-1}, e_i)$ for some functions $\alpha_i$, where $G_0\subsetneq G_1\subsetneq\dots\subsetneq G_N$ is the denoted process with non-descending order of edge weights (that is, $w(e_1)\leq w(e_2)\leq\dots\leq w(e_N))$. For both algorithms, we find the asymptotic value of $\EE\alpha_i(G_{i-1}, e_i)$ and then establish $\sum\limits_{i=1}^{cn}\alpha_i(G_{i-1}, e_i)$ for $c\in\mathbb{R}_+$. Further, we show that $\sum\limits_{i=1}^{cn}\alpha_i(G_{i-1}, e_i)$ tends to the result value of the respective algorithms as $c\to\infty$ and by that finish the proof.

\section{The algorithms}
Recall the random $t$-uniform hypergraph process with ascending order of edge weights \[G_0\subsetneq G_1\subsetneq\dots\subsetneq G_N,N=\binom{n}{t}.\]
By $G_t(n, m)$ we denote $G_m$. If $t=2$, we omit the index: $G(n,m):=G_2(n,m)$.

Let $w_i$ denote $w(e_i)$.
Standard implementation of Kruskal algorithm (see \cite{kruskal1956shortest}) implies \[\alpha_1(G_{i-1},e_i):=w_i\cdot\Ind(e_i\textnormal{ connects different components in } G_{i-1})\footnote{For any event $A$, by $\Ind(A)$ we denote a random variable which equals 1 if $A$ occurs and 0 otherwise.}.\]
If we will perform this exact algorithm on a $\tilde{\mathfrak{G}}_n$, the result value will be the sum of the weights of all the edges $e_i$ that connect different components in their respective hypergraphs $G_{i-1}$. These edges form a connected hypergraph and hence $\sum \alpha_1(G_{i-1},e_i)\geq \textnormal{MST}(\tilde{\mathfrak{G}}_n)$ (by definition the minimum spanning subgraph is not greater in sum of weights than any other).

Now consider $E'$ --- the set of edges of the MST($\tilde{\mathfrak{G}}_n$). By replacing every edge $e_i$ with any tree on the same set of vertices with edge weights $w_i$ we obtain a connected weighted multigraph $G'$ with MST not greater than $(t-1)\textnormal{MST}(\tilde{\mathfrak{G}}_n)$
Consider a multigraph $G''$ obtained via replacing every edge of $\tilde{\mathfrak{G}}_n$ $e_i$ with a \textit{clique} on the same set of vertices with edge weights $w_i$. $G'$ is a subgraph of $G''$, hence
\[\textnormal{MST}(G'')\leq \textnormal{MST}(G')\leq(t-1)\textnormal{MST}(\tilde{\mathfrak{G}}_n).\]

Note that $\textnormal{MST}(G'')$ can be calculated by a greedy algorithm on $\tilde{\mathfrak{G}}_n$ with $\alpha_2(G_{i-1}, e_i)=w_i\cdot (K_i-1)$, where $K_i$ is the number of components that $e_i$ connects in $G_{i-1}$.

\section{Algorithms value on prefixes}
For a sequence of events $\{\xi_i\}_{i=1}^{+\infty}$, we say that $\xi_i$ occurs \textit{w.h.p.} (with high probability) if $\lim\limits_{i\to\infty}\Pr[\xi_i]=1$.

Consider the binomial model of a random $t$-uniform hypergraph $G_t(n,p)$ (that is, every edge is drawn independently with probability $p$) with $\displaystyle\frac{\binom{n}{t}p}{n}\to c$. As it was shown in \cite{Behrisch2009}, if $c<1/(t(t-1))$ then w.h.p. all the components of $G_t(n,p)$ are of the size $O(\log n)$. On the other hand, if $c>1/(t(t-1))$ then w.h.p. all the components but one are of size $O(\log n)$ with the biggest one (also known as the \textit{giant component}) being of size $S$, $S/n\overset{\Pr}{\to}\beta\in(0,1)$. It was also shown than $\beta$ satisfies
\begin{equation}\label{eq:beta}
    c=\frac{\ln(1-\beta)}{t\left((1-\beta)^{t-1}-1\right)}.
\end{equation}

We can easily extend a slightly weaker version of these results to the uniform model $G_t(n, m)$: fix $G_t(n, p)$ such that $\binom{n}{t}p = cn - \sqrt{n}\log n$. W.h.p. the number of edges of this hypergraph is between $cn - 2\sqrt{n}\log n$ and $cn$ edges and its components suffice the conditions above, which means that we can draw at most $2\sqrt{n}\log n$ additional edges to achieve $G_t(n, m)$ with asymptotically similar size of giant component (if present) and sizes of smaller components at most $O(\sqrt{n}\log^2 n)$. Let's call a component \textit{small} if it isn't the giant component of the graph.

\begin{lemma}\label{lem:k}
Fix $c\in \mathbb{R}_+\textbackslash\{1/(t(t-1))\}$ and let $i=\lceil cn\rceil$. For the first algorithm, 
\[\frac{\EE\alpha_1(G_{i-1}, e_i)}{\EE w_i}\to1-\beta^t(c).
\]
For the second algorithm, 
\[\frac{\EE\alpha_2(G_{i-1}, e_i)}{\EE w_i}\to(1-\beta(c))t-(1-\beta^t(c)),\]
where $\beta(c)$ is equal to 0 if $c<1/t(t-1)$ and the unique solution of (\ref{eq:beta}) otherwise.
\end{lemma}
\begin{proof}
    Let $K_i$ denote the number of components that $e_i$ connects in $G_{i-1}$. Then $\alpha_1(G_{i-1}, e_i)=w_i\cdot \Pr[K_i>1]$ and $\alpha_2(G_{i-1}, e_i)=w_i\cdot(K_i-1)$. Let $e$ be the uniformly random edge of $\tilde{\mathfrak{G}}_n$, then for any fixed $e_1,e_2,\dots,e_{i-1}$ $\sup\limits_{A\subset2^{E(\tilde{\mathfrak{G}}_n)}}|\Pr(e\in A)-\Pr(e_i\in A)|<\frac{ct!}{n^{t-1}}$. Let $K'_i$ denote the number of components that $e$ connects in $G_{i-1}$. Note that since w.h.p. all the small components in $G_{i-1}$ are of size $O(\sqrt{n}\log^2 n)$, probability that there are at least two vertices of $e$ that lie in the same small component is at most $O(\log^4 n/n)$, which means that
    \[\Big|\Pr[K'_i=k]-\Pr[\textnormal{exactly $t-k+1$ vertices of $e$ lie in the giant component}]\Big|=o(1)\]
    for any $k>1$ and
    \[\Big|\Pr[K'_i=1]-\Pr[\textnormal{at least $t-1$ vertices of $e$ lie in the giant component}]\Big|=o(1).\]

    The number of vertices of $e$ in the giant component converges to $\Binomial(n, \beta(c))$ by distribution, the same holds for $e_i$. This number is independent with $w_i$, which finishes the proof.
\end{proof}

Using Azuma-Hoeffding inequality \cite{A1967}, one can easily prove that for any $\varepsilon>0$ w.h.p. the following holds:

\begin{lemma}\label{lem:a}
    Fix $0<c_0<c_1<\dots<c_k$. By $\mathfrak{a}^i_j$ denote $\EE\alpha_i(G_{\lceil c_jn\rceil-1}, e_{\lceil c_jn\rceil})$. Then for any $\varepsilon>0, \ell\in\{1,2\}$, w.h.p. the following holds:
    \[\sum_{i=\lceil c_jn\rceil+1}^{\lceil c_{j+1}n\rceil}\alpha_\ell(G_{i-1}, e_i)\in\left((1-\varepsilon)(c_{j+1}-c_j)n\frac{\mathfrak{a}^\ell_{j+1}\EE w_{c_jn}}{\EE w_{c_{j+1}n}}, (1+\varepsilon)(c_{j+1}-c_j)n\frac{\mathfrak{a}^\ell_j\EE w_{c_{j+1}n}}{\EE w_{c_jn}}\right).\]
\end{lemma}

Using this fact one can deduce that for any $c>1/(t(t-1))$ and $i\in\{1,2\}$,
\[\EE\sum_{i=1}^{\lceil cn\rceil}\alpha_\ell(G_{i-1}, e_i)\to a\sqrt[t-1]{t!}\left(\int\limits_{0}^{\frac{1}{t(t-1)}}\sqrt[t-1]{x}\diff x+\int\limits_{\frac{1}{t(t-1)}}^c\sqrt[t-1]{x} \cdot f_\ell(\beta(x))\diff x\right),\]
where $f_1(x)=1-x^t$ for the first algorithm and $f_2(x)=t(1-x)-(1-x)^t$ for the second one (since the sum on the left is transformed to the Darboux sum for the integral in the right part using Lemma \ref{lem:a}) Proof of Lemma \ref{lem:a} is given in the appendix.

\section{Completion of the proof}

Define $A_c=\sum\limits_{i=1}^{\lceil cn\rceil}\alpha(G_{i-1}, e_i)$. Since we established the limit value of $A_c$ as $n$ tends to infinity, we now need to prove that $\lim\limits_{c\to\infty}A_c=A(G)$. First we note that w.h.p. $\sum\limits_{i=1}^{2n\ln n}\alpha(G_{i-1}, e_i)=A(G)$ since a random hypergraph with such many drawn edges in connected w.h.p., as it was shown by Poole \cite{Poole2015}.

\begin{lemma}\label{lem:b}
    Let $C(G)$ denote the number of connectivity components in $G$. Then for any $m\in[0, 2n\ln n]$, the following holds:
    \[\Pr\left[C(G_t(n,m))>n\cdot f(m/n)+n^{4/5}\right]=O(n^{-1/6}),\]
    where $f$ is an exponentially decreasing function, that is $f(x)=O\left(\exp\left(-Cx\right)\right)$ for some $C>0$.
\end{lemma}
Proof of Lemma \ref{lem:b} is given in the appendix.

Let $C(n, c)$ denote $C(G_t(n, cn))$. Lemma \ref{lem:b} implies that for any $\varepsilon>0$, w.h.p.
\begin{multline*}
    \sum\limits_{i=\lceil cn\rceil}^{2n\ln n}\alpha(G_{i-1}, e_i)\leq(1+\varepsilon)\sum\limits_{i=\lceil c\rceil}^{2\ln n} a\frac{\sqrt[t-1]{it!}}{n}\left(C(n, i-1)-C(n, i)\right)=a(1+\varepsilon)\cdot\\
    \cdot\left(\frac{\sqrt[t-1]{(c-1)t!}}{n}C(n, c-1)-\frac{\sqrt[t-1]{2\ln nt!}}{n}C(n, 2 \ln n)+\sum_{i=\lceil c\rceil}^{2\ln n - 1}C(n, i)\cdot(\sqrt[t-1]{i+1}-\sqrt[t-1]{i})\frac{\sqrt[t-1]{t!}}{n}\right)\leq\\
    \leq a(1+\varepsilon)\cdot\left(\underbrace{\sqrt[t-1]{(c-1)t!}f(c-1)}_{O(c\cdot f(c))}+\underbrace{\sum_{i=\lceil c\rceil}^{2\ln n - 1}f(i)\cdot(\sqrt[t-1]{(i+1)t!}-\sqrt[t-1]{it!})}_{O(f(c))}+n^{4/5}\cdot O(\ln^2n/n)\right)=\\
    =O(c\cdot\exp(-Cc))+O(\ln^2n/n^{1/5}).
\end{multline*}

This completes the proof of Theorem \ref{th:main}, as we now have obtained two bounds, which can be reformulated as it is suggested in the introduction. 

\bibliographystyle{plainurl}
\bibliography{bibl}
\appendix
\newpage

\section*{Proof of Lemma \ref{lem:a}}
% Suppose that we have $G_t(n, \lceil c_jn\rceil)$. Let us add $\lceil c_{j+1}n\rceil-\lceil c_jn\rceil$ random independent edges chosen with uniform distribution on all the $\binom{n}{t}$ edges. With probability $1-O(1/n)$ every new added edge is unique (no two edges coincide). 

Recall $K_i$ from the proof of Lemma \ref{lem:k}. It is easy to see that $\EE K_i$ is non-increasing, as well as $\Pr[K_i>1]$. Let $\alpha_\ell(G_{i-1}, e_i)=w_i\cdot\Ind^\ell_i$, where $w_i$ is the weight of $e_i$ and $\Ind^\ell_i$ is some non-decreasing function of $K_i$ ($\Ind^1_i = \min(K_i, 1), \Ind^2_i = K_i$). Since $w_i$ is non-decreasing, the same holds for $\EE w_i$, whereas $\EE \Ind^\ell_i$ is non-increasing.

Consider $F(e_{\lceil c_jn\rceil+1},\dots,e_{\lceil c_{j+1}n\rceil})=\sum\limits_{i=\lceil c_jn\rceil+1}^{\lceil c_{j+1}n\rceil}\Ind^\ell_i$. For $\lceil c_{j+1}n\rceil-\lceil c_jn\rceil$ iid edges, uniformly distributed over $E(\tilde{\mathfrak{G}}_n)$, w.h.p none of them coincide with each other or with any of $\lceil c_jn\rceil$ edges which are already drawn, which means that we can substitute $e_{\lceil c_jn\rceil+1},\dots,e_{\lceil c_{j+1}n\rceil}$ with iid uniformly distributed edges. Changing only one variable can't result in more than $t-1$ change in the value of $F$, and therefore by Azuma-Hoeffding inequality \cite{A1967} for any $\varepsilon>0$ w.h.p. $|F-\EE F|<\varepsilon n$. On the other hand by Kolmogorov lemma \cite{ref1} for every $\varepsilon>0$ w.h.p. $\left|w_{c_j}-a\frac{\sqrt[t-1]{c_jt!}}{n}\right|<\varepsilon/n$. We can derive the following bound on $\EE F$:
\[(c_{j-1}-c_j)n\EE\Ind^\ell_{c_j}\leq\EE F\leq(c_{j-1}-c_j)n\EE\Ind^\ell_{c_{j-1}}.\]
Therefore, for every $\varepsilon>0$ w.h.p. 
\[\sum\limits_{i=\lceil c_jn\rceil+1}^{\lceil c_{j+1}n\rceil}\alpha_\ell(G_{i-1}, e_i)\leq(c_{j-1}-c_j)n\EE\Ind^\ell_{\lceil c_{j-1} n\rceil}\cdot\EE w_{\lceil c_jn\rceil}\cdot(1+\varepsilon).\]

\section*{Proof of Lemma \ref{lem:b}}
    
Independently draw $m\leq2n\ln n$ edges $e^{(i)}=(e^{(i)}_1,\dots,e^{(i)}_t)$ in graph $G_t$ on $n$ vertices (assume that $(e^{(i)}_1,\dots,e^{(i)}_t)$ are randomly shuffled, so that every permutation is equally probable). With probability $1-O(\ln^2n/n^{t-2})$ no two drawn edges coincide. Then construct a graph $G$ on $n$ vertices by doing the following: for every drawn (hyper-) edge $e^{(i)}$ let's draw an edge $e^i=(e^{(i)}_1, e^{(i)}_2)$. Assume that after deleting the duplicates we are left with $E_0$ edges in $G$, then for large $n$
\[\EE E_0=\binom{n}{2}\cdot\left(1-\left(1-\frac{1}{\binom{n}{2}}\right)^m\right)\geq\binom{n}{2}\cdot\left(1-\exp\left(-m/\binom{n}{2}\right)\right)\geq m-\frac{m^2}{\binom{n}{2}}.\]

Given that $E_0\leq m$ a.s., we derive that $\Pr(E_0<m/2)=O(\ln n/n)$. If $E_0\geq m/2$, randomly delete some of the edges to get to exactly $m/2$ edges.

At the moment, we managed to get a graph $G(n, m/2)$, the edges of which are subedges of the original hypergraph. It follows that $C(G(n, m/2))\geq C(G_t(n, m))$. As it was shown by Frieze, there is some $f$, such that for any $n$ and $m\leq 2n\ln n$ the following holds:
\[\Pr(C(G(n, m/2))>n\cdot f(m/2n)+n^{4/5})=O(n^{-1/6}).\]

Let $\beta(c)$ denote the limit of fraction of vertices in a giant component in a $G(n, cn/2)$ as $n$ tends to $+\infty$. As it shown in \cite{FK2015}, the recursive equation $\beta(c)+e^{-c\beta(c)}=1$ holds for $c>1$. Note that $\beta(c)+f(c/2)\leq 1$ because $n\cdot \beta(c)(1+o(1))$ vertices lie in the same component in a graph $G(n, cn/2)$, and there are also $n\cdot f(c/2)(1+o(1))$ other components, each of which contain at least one vertex, and the amount of vertices is at most $n\geq n(\beta(c)+f(c/2))$. From the equation for $\beta$, one can see that $1-\beta$ is exponentially decreasing and so is $f$ since $f(x)\leq 1-\beta(2x)$.
Therefore $C(G_t(n,m))\leq C(G_t(n, m/2)\leq 1-\beta(m/n)$ with probability $1-O(n^{-1/6})$ --- by that, Lemma \ref{lem:b} is proved.
\end{document}